\documentclass[reqno]{amsart}
\usepackage{amssymb, latexsym, amsthm, enumitem, color, amsmath, tikz-cd, mathtools}

\usepackage[unicode=true,pdfusetitle,bookmarks=true,bookmarksnumbered=false,
bookmarksopen=false,breaklinks=true,pdfborder={0 0 0},
pdfborderstyle={},backref=false,colorlinks=true]{hyperref}
\definecolor{myblue}{rgb}{0.09,0.32,0.44} %
\hypersetup{pdfborder={0 0 0},pdfborderstyle={},colorlinks=true,linkcolor=myblue,citecolor=myblue,urlcolor=blue}

\newtheorem{thm}{Theorem}[section] %
\newtheorem*{thm*}{Theorem}

\newtheorem{cor}[thm]{Corollary}

\newtheorem{fact}[thm]{Fact}
\newtheorem{lem}[thm]{Lemma}
\newtheorem{prop}[thm]{Proposition}

\theoremstyle{remark}
\newtheorem{rem}[thm]{Remark}

\newtheorem*{rem*}{Remark}
\newtheorem*{rems*}{Remarks}

\newcommand\Cref[1]{{Corollary~\ref{#1}}}

\newcommand\Fref[1]{{Fact~\ref{#1}}}

\newcommand\Pref[1]{{Proposition~\ref{#1}}}
\newcommand\Rref[1]{{Remark~\ref{#1}}}
\newcommand\Tref[1]{{Theorem~\ref{#1}}}

\newcommand{\Z}{\mathbb{Z}}

\newcommand{\C}{\mathbb{C}}
\newcommand{\FF}{\mathbb{F}}
\newcommand{\Qu}{\mathcal{C}}

\newcommand{\set}[1]{\left\{#1\right\}}
\newcommand{\sub}{\subseteq}

\newcommand{\sg}[1]{\left\langle #1\right\rangle}
\newcommand{\suchthat}{\,\ifnum\currentgrouptype=16\middle\fi|\,}

\newcommand{\ab}[1]{#1^{\mathrm{ab}}}
\newcommand{\prof}[1]{\widehat{#1}}
\newcommand{\LL}{\mathcal{L}}

\newcommand\PSL[1][d]{{\operatorname{PSL}_{#1}}} %

\makeatletter
\def\moverlay{\mathpalette\mov@rlay}
\def\mov@rlay#1#2{\leavevmode\vtop{%
   \baselineskip\z@skip \lineskiplimit-\maxdimen
   \ialign{\hfil$\m@th#1##$\hfil\cr#2\crcr}}}
\newcommand{\charfusion}[3][\mathord]{
    #1{\ifx#1\mathop\vphantom{#2}\fi
        \mathpalette\mov@rlay{#2\cr#3}
      }
    \ifx#1\mathop\expandafter\displaylimits\fi}
\makeatother

\newlength{\tempindent}
\newcommand{\lazyenum}{
\setlength{\tempindent}{\parindent}
\begin{enumerate}[leftmargin=0cm,itemindent=0.7cm,labelwidth=\itemindent,labelsep=0cm,align=left,label=(\arabic*)]
\setlength{\parskip}{\smallskipamount}
\setlength{\parindent}{\tempindent}
}

\title{Profinite rigidity of lamplighter groups}
\author{}
\author{Guy Blachar}
\address{Bar-Ilan University, Ramat Gan 5290001, Israel}
\address{The Weizmann Institute of Science, Rehovot 7610001, Israel}
\email{guy.blachar@gmail.com}

\makeatletter
\@namedef{subjclassname@2020}{2020 Mathematics Subject Classification}
\makeatother
\subjclass[2020]{20E18, 20E22, 13C05}
\keywords{Profinite rigidity, lamplighter groups, wreath products.}

\begin{document}

\begin{abstract}
  We show that the lamplighter groups $(\Z/p\Z)^n\wr\Z$, where~$p$ is prime and $n\ge 1$ is a positive integer, are profinitely rigid.
\end{abstract}

\maketitle

\section{Introduction}

Let $\Gamma$ be a finitely generated residually finite group. The profinite completion of $\Gamma$ is the inverse limit of the inverse system of its finite quotients, i.e.\ $\prof{\Gamma}=\varprojlim_{N\vartriangleleft_f\Gamma}\Gamma/N$. One may ask the following question: is $\Gamma$ defined (up to isomorphism) by its profinite completion? In other words, given such a group $\Gamma$, we ask whether for any finitely generated residually finite group $\Delta$, an isomorphism $\prof{\Gamma}\cong\prof{\Delta}$ implies $\Gamma\cong\Delta$. Such groups are called \textbf{profinitely rigid}.

An equivalent formulation can be given in terms of the finite quotients of $\Gamma$. We let $\Qu(\Gamma)$ denote the set of isomorphism classes of all finite quotients of $\Gamma$. Note that for any finitely generated residually finite group $\Gamma$ we have $\Qu(\Gamma)=\Qu(\prof{\Gamma})$, so $\prof{\Gamma}\cong\prof{\Delta}$ implies $\Qu(\Gamma)=\Qu(\Delta)$. The other direction holds as well, by combining results of Dixon, Formanek, Poland and Ribes \cite{DFPR} and of Nikolov and Segal \cite{NS}. Therefore, a group $\Gamma$ is profinitely rigid if and only if $\Qu(\Gamma)=\Qu(\Delta)$ implies $\Gamma\cong\Delta$.\medskip

Any finitely generated abelian group is profinitely rigid. Other examples of profinitely rigid groups, studied by Bridson, McReynolds, Reid and Spitler, include some arithmetic lattices in $\PSL[2](\C)$ (see \cite{BMRS20}) and some triangle groups (see \cite{BMRS21}). Recently, works of Corson, Hughes, M{\"o}ller and Varghese \cite{CorsonHughesMollerVarghese24} and of Paolini and Sklinos~\cite{PaoliniSklinos24} show that affine Coxeter groups are profinitely rigid. In contrast, there are examples of non-isomorphic metacyclic groups $\Gamma_1,\Gamma_2$ for which $\prof{\Gamma_1}\cong\prof{\Gamma_2}$ (see the main theorem in~\cite{Ba}), and other examples of non-profinitely rigid groups such as in~\cite{CLRS}. We refer the reader to \cite{Re} for a survey on profinite rigidity.\medskip

The classical lamplighter group $\LL=(\Z/2\Z)\wr\Z$ and similar wreath products play an important role in group theory, as well as in the study of random walks. The subgroup lattice and profinite completion of $\LL$, and more generally of the groups $\LL_{n,p}=(\Z/p\Z)^n\wr\Z$ for a prime number $p$ and positive integer $n\ge 1$, were studied by Grigorchuk and Kravchenko~\cite{GK}. It is therefore natural to study these groups also in the context of profinite rigidity.\medskip

Our goal in this paper is to prove the following theorem:

\begin{thm}\label{thm:main}
  The group $\LL_{n,p}=(\Z/p\Z)^n\wr\Z$ is profinitely rigid for any prime $p$ and any positive integer $n\ge 1$.
\end{thm}

The idea of the proof is as follows. Let $\Gamma$ be a finitely generated residually finite. If $\Qu(\Gamma)=\Qu(\LL_{n,p})$, the abelianizations of $\Gamma$ and $\LL_{n,p}$ must be isomorphic, i.e., we have $\ab{\Gamma}\cong\ab{\LL_{n,p}}\cong(\Z/p\Z)^n\times\Z$. We use this to get a semidirect product decomposition of a (quotient of) $\Gamma$ as $N\rtimes\Z$, and study the action of~$\Z$ on~$N$. We show that $N$ can be viewed as a finitely generated $\FF_p[x^{\pm 1}]$-module. We then utilize the structure theory of modules over $\FF_p[x^{\pm 1}]$, which is a principal ideal domain, to show that~$N$ can then be projected to the subgroup $\bigoplus_{\Z}(\Z/p\Z)^n$ of $\LL_{n,p}$ consistently with the $\Z$-actions. This induces a projection $\Gamma\to\LL_{n,p}$, and we conclude $\Gamma\cong\LL_{n,p}$ by \cite[Theorem 3]{DFPR}.

\section{\texorpdfstring{The groups $\LL_{n,p}$}{The groups L_n,p}}

In this section we recall some properties of the groups $\LL_{n,p}$.

Let $G,H$ be groups. The \textbf{(full) wreath product} (or \textbf{lamplighter group}) $H\wr G$ to be the semidirect product $\left(\bigoplus_G H\right)\rtimes G$, where~$G$ acts on $\bigoplus_G H$ by shifting the index, i.e.\ $g\cdot(h_a)_{a\in G}=(h_{g^{-1}a})$. We write elements of $G\wr H$ as pairs $(L,g)$, where $L\in\bigoplus_G H$ and $g\in G$. The group $G$ is called the \textbf{base group} of $H\wr G$, and the group $H$ is called the \textbf{lamp group} of~$H\wr G$. In particular, the group $H\wr G$ is generated by the copy of $G$ as its base group and one copy of $H$ as a direct summand in $\bigoplus_G H$.

The abelianization of lamplighter groups is given by the following claim:

\begin{prop}\label{prop:lamp-abel}
  For any two groups $G,H$ we have $\ab{(H\wr G)}=\ab{H}\times\ab{G}$.
\end{prop}

\begin{proof}
  Write $K=\bigoplus_G H$, so that $H\wr G=K\rtimes G$. Then $\ab{(H\wr G)}=(\ab{K})_G\times\ab{G}$ where $\ab{K}=\bigoplus_G\ab{H}$ and $(\ab{K})_G=\ab{K}/\sg{gk-g\suchthat g\in G,k\in\ab{K}}=\ab{H}$, so the claim follows.
\end{proof}

Our groups of interest are $\LL_{n,p}=(\Z/p\Z)^n\wr\Z$ where $p$ is prime and $n\ge 1$ is a positive integer. The above proposition implies:

\begin{cor}\label{cor:Lnp-abel}
  $\ab{\LL_{n,p}}\cong(\Z/p\Z)^n\times\Z$.
\end{cor}

We now describe the finite quotients of $\LL_{n,p}$, in the spirit of \cite{GK}.

Let $\varphi\colon\LL_{n,p}\to F$ be an epimorphism where $F$ is a nontrivial finite group. Take the minimal $m\ge 1$ such that $(\boldsymbol{0},m)\in\ker\varphi$ (here $\boldsymbol{0}$ denotes the lamp configuration where all lamps are set to $0$). The kernel of $\varphi$ must contain the normal subgroup~$H_m$ generated by $(\boldsymbol{0},m)$, and thus $\varphi$ factors through $\LL_{n,p}/H_m\cong(\Z/p\Z)^n\wr(\Z/m\Z)$. %
Conversely, any quotient of $(\Z/p\Z)^n\wr(\Z/m\Z)$ is a quotient of~$\LL_{n,p}$. This proves the following:

\begin{cor}\label{cor:Qu-Lnp}
  For any prime $p$ and any $n\ge 1$ we have
  \[
    \Qu(\LL_{n,p})=\bigcup_{m=1}^{\infty}\Qu\big((\Z/p\Z)^n\wr(\Z/m\Z)\big).
  \]
\end{cor}

For this reason all of our arguments below will only use homomorphisms to finite lamplighters of the above form.

\section{Proof of the main theorem}

Let $\Gamma$ be a finitely generated residually finite group with $\Qu(\Gamma)=\Qu(\LL_{n,p})$. Our goal is to prove that $\Gamma\cong\LL_{n,p}$. We do this in several steps.

\subsection{A semidirect product decomposition}

In the first step of the proof, we find a quotient~$\Gamma_0$ of $\Gamma$ such that $\Gamma_0=N\rtimes\Z$ for an appropriate subgroup $N\le\Gamma_0$. Although one can show that  $\Gamma$ splits as a semidirect product of this form as well, it will be easier to work with the quotient $\Gamma_0$ we will present.

\begin{prop}\label{prop:gamma-abel}
  The abelianization of $\Gamma$ is $\ab{\Gamma}\cong(\Z/p\Z)^n\times\Z$.
\end{prop}

\begin{proof}
  It is known that if $\Qu(\Gamma)=\Qu(\Delta)$, then $\ab{\Gamma}\cong\ab{\Delta}$ (see e.g.\ \cite[Remark~ 3.2]{Re}); therefore, \Cref{cor:Lnp-abel} shows that $\ab{\Gamma}\cong\ab{\LL_{n,p}}\cong(\Z/p\Z)^n\times\Z$.
\end{proof}

Denote by $\mathcal{A}$ the set of all homomorphisms $\Gamma\to(\Z/p\Z)^n\wr(\Z/m\Z)$ for all positive integers~$m\ge 1$ satisfying $p\nmid m$. Define
\[
  \Gamma_0=\Gamma/\bigcap_{\varphi\in\mathcal{A}}\ker\varphi,
\]
which is a finitely generated residually finite group. By the definition of~$\Gamma_0$, for any $1\ne a\in\Gamma_0$ there are a positive integer $m\ge 1$ with $p\nmid m$ and a homomorphism $\varphi\colon\Gamma_0\to(\Z/p\Z)^n\wr(\Z/m\Z)$ such that $\varphi(a)\ne 1$. Writing $\mathcal{A}_0$ for the set of such homomorphisms, we have
\begin{equation}\label{eq:A0-kernel-cap}
  \bigcap_{\varphi\in\mathcal{A}_0}\ker\varphi=1.
\end{equation}

\begin{cor}
  The abelianization of $\Gamma_0$ is $\ab{\Gamma_0}\cong(\Z/p\Z)^n\times\Z$.
\end{cor}

\begin{proof}
  The natural projection $\Gamma\to\Gamma_0$ translates to a projection $\ab{\Gamma}\to\ab{\Gamma_0}$, so by \Pref{prop:gamma-abel} there is a projection $(\Z/p\Z)^n\times\Z\to\ab{\Gamma_0}$. Also, by the choice of~$\Gamma_0$, for any $m\ge 1$ with $p\nmid m$ there is a projection $\Gamma_0\to(\Z/p\Z)^n\wr(\Z/m\Z)$, which translates to a projection $\ab{\Gamma_0}\to(\Z/p\Z)^n\times(\Z/m\Z)$. Therefore we must have $\ab{\Gamma_0}\cong(\Z/p\Z)^n\times\Z$.
\end{proof}

Let $\pi\colon\Gamma_0\to\ab{\Gamma_0}\cong(\Z/p\Z)^n\times\Z$ denote the natural projection to the abelianization, and let $\pi'\colon\Gamma_0\to\Z$ denote the projection of $\pi$ to its $\Z$-component. Since~$\Z$ is free, we have:

\begin{cor}
  We have $\Gamma_0=N\rtimes\Z$ for $N=\ker\pi'$.
\end{cor}

\subsection{\texorpdfstring{The structure of $N$}{The structure of N}}

Our next step is understanding the properties of $N$ and the action of $\Z$ on $N$. To do this, we first study the images of $N$ under the maps in $\mathcal{A}_0$.

\begin{prop}\label{prop:image-hom-A0}
  Let $\varphi\colon\Gamma_0=N\rtimes\Z\to(\Z/p\Z)^n\wr(\Z/m\Z)$ be a homomorphism, where $p\nmid m$. Then $\varphi(N)\sub\bigoplus_{\Z/m\Z}(\Z/p\Z)^n$.
\end{prop}

\begin{proof}
  Let $\rho\colon(\Z/p\Z)^n\wr(\Z/m\Z)\to(\Z/p\Z)^n\times(\Z/m\Z)$ denote the abelianization map. As $\rho\circ\varphi$ maps $\Gamma_0$ to an abelian group, it factors through the abelianization $\ab{\Gamma_0}\cong(\Z/p\Z)^n\times\Z$, i.e.\ there exists $\psi\colon(\Z/p\Z)^n\times\Z\to(\Z/p\Z)^n\times(\Z/m\Z)$ such that the following diagram commutes:
  \begin{center}
    \begin{tikzcd}
      \Gamma_0\arrow[r,"\varphi"]\arrow[d,"\pi"] & (\Z/p\Z)^n\wr(\Z/m\Z)\arrow[d,"\rho"]\\
      (\Z/p\Z)^n\times\Z\arrow[r,"\psi"] & (\Z/p\Z)^n\times(\Z/m\Z).
    \end{tikzcd}
  \end{center}

  Note that the assumption $p\nmid m$ implies that $\psi$ maps $(\Z/p\Z)^n$ to itself. The subgroup $N\le\Gamma_0$ is mapped by $\pi$ to $(\Z/p\Z)^n$, and thus it is mapped to $(\Z/p\Z)^n$ by $\psi\circ\pi=\rho\circ\varphi$. By the definition of $\rho$, it follows that $\varphi(N)\sub\bigoplus_{\Z/m\Z}(\Z/p\Z)^n$, as required.
\end{proof}

\begin{cor}
  The group $N$ is abelian and of exponent $p$.
\end{cor}

\begin{proof}
  Using \eqref{eq:A0-kernel-cap} and \Pref{prop:image-hom-A0}, the group $N$ is residually finite, and all of its nontrivial elements are nontrivial in some projection to $\bigoplus_{\Z/m\Z}(\Z/p\Z)^n$ for some~$m\ge 1$. The latter group is abelian and of exponent~$p$ for all $m$, which proves the corollary.
\end{proof}

It follows from the previous corollary that $N$ can be seen as an $\FF_p$-vector space. For this reason, we use an additive notation for the group operation of $N$. We write elements of $\Gamma_0$ as pairs $(a,k)$, where $a\in N$ and $k\in\Z$. In addition, $\Gamma_0=N\rtimes\Z$, so there is an action of $\Z$ on $N$ by automorphisms. We conclude that:

\begin{cor}\label{cor:N-mod}
  The group $N$ is a finitely generated $\FF_p[x^{\pm 1}]$-module, where the action of $x$ is induced from the action of $\Z$ on $N$ in $\Gamma_0$.
\end{cor}

\begin{proof}
  This follows from the fact that $\FF_p[\Z]\cong\FF_p[x^{\pm 1}]$. As $\Gamma_0$ is finitely generated as a group, it follows that $N$ is finitely generated as an $\FF_p[x^{\pm 1}]$-module.
\end{proof}

From now on, let $R_p=\FF_p[x^{\pm 1}]$ denote the ring of Laurent polynomials in the variable $x$ over $\FF_p$. We recall some known properties of $R_p$; these facts are standard, and we refer the reader to \cite[Lemma 2.1]{GK} for proofs of these claims.

\begin{fact}\label{fact:R-facts}
  \begin{enumerate}
    \item The ring $R_p$ is a principal ideal domain.
    \item For any $0\ne f\in R_p$, the ring $R_p/fR_p$ is finite. If $f$ is a (standard) polynomial with $f(0)\ne 0$, then $\left|R_p/fR_p\right|=p^{\deg f}$.
  \end{enumerate}
\end{fact}

Combining \Cref{cor:N-mod} with \Fref{fact:R-facts}, we see that $N$ is a finitely generated module over a principal ideal domain, and thus we may deduce the following:

\begin{cor}
  There exist $r,s\ge 0$ and $0\ne f_1,\dots,f_s\in R_p$ such that $f_1|\cdots|f_s$ and
  \begin{equation}\label{eq:N-decomp}
    N\cong R_p^r\times R_p/f_1R_p\times\cdots\times R_p/f_sR_p
  \end{equation}
  as $R_p$-modules. We further assume that $f_i\in\FF_p[x]$ and $f_i(0)\ne 0$, by multiplying~$f_i$ with an appropriate power of $x$.
\end{cor}

\begin{rem}\label{rem:sem-Rn-Z}
  As an $\FF_p$-vector space, $R_p=\bigoplus_{i=-\infty}^{\infty}\FF_px^i$, and the action of $x$ on~$R_p$ (by multiplication) induces a shift of the direct summands. Therefore, letting $\Z$ act on~$R_p^n$ by the action of $x$, we have $R_p^n\rtimes\Z\cong\LL_{n,p}$.
\end{rem}

\subsection{\texorpdfstring{Determining the module structure of $N$}{Determining the module structure of N}}

Our next goal is to prove that, as $R_p$-modules, there is an epimorphism $N\to R_p^n$ (or equivalently that $r\ge n$).

Fix $m\ge\sum_{i=1}^s\deg f_i+2$ such that $p\nmid m$. Let $\varphi\colon\Gamma_0\to(\Z/p\Z)^n\wr(\Z/m\Z)$ be a projection. Applying \Pref{prop:image-hom-A0}, for any $a\in N$ and $k\in\Z$ we have
\[
  \varphi(a,k)=\varphi((a,0)(0,k))=\varphi(a,0)\varphi(0,k)=(\psi(a)+g(k),f(k))
\]
for some group homomorphisms $\psi\colon N\to(\Z/p\Z)^n\wr(\Z/m\Z)$ and $f\colon\Z\to\Z/m\Z$, and a function $g\colon\Z\to(\Z/p\Z)^n\wr(\Z/m\Z)$ with $g(0)=0$. Since $\varphi$ is onto, $f$ must also be onto, and thus $f(1)\in(\Z/m\Z)^{\times}$. By composing $\varphi$ with an automorphism of $(\Z/p\Z)^n\wr(\Z/m\Z)$, we may assume that $f(1)=1$, i.e.\ $f(k)=k$ for all $k\in\Z$.

The action of $\Z/m\Z$ on $\bigoplus_{\Z/m\Z}(\Z/p\Z)^n$ induced from the lamplighter structure endows $\bigoplus_{\Z/m\Z}(\Z/p\Z)^n$ with an $R_p$-module structure, which is isomorphic to $(R_p/(x^m-1)R_p)^n$ as $R_p$-modules (in a similar manner to \Rref{rem:sem-Rn-Z}). Under this operation, we note the following:

\begin{lem}
  The function $\psi\colon N\to\bigoplus_{\Z/m\Z}(\Z/p\Z)^n$ is a homomorphism of $R_p$-modules.
\end{lem}

\begin{proof}
  It is enough to verify that for any $a\in N$ we have $\psi(xa)=x\psi(a)$. Indeed, for any $a\in N$
  \[
    \varphi((0,1)(a,0))=\varphi(xa,1)=(\psi(xa)+g(1),1)
  \]
  whereas
  \[
    \varphi(0,1)\varphi(a,0)=(g(1),1)(\psi(a),0)=(g(1)+x\psi(a),1),
  \]
  and the result follows.
\end{proof}

\begin{lem}
  For any $k,k'\in\Z$ the function $g$ satisfies $g(k+k')=g(k)+x^kg(k')$.
\end{lem}

\begin{proof}
  We have
  \[
    \varphi((0,k)(0,k'))=\varphi(0,k+k')=(g(k+k'),k+k')
  \]
  while
  \[
    \varphi(0,k)\varphi(0,k')=(g(k),k)(g(k'),k')=(g(k)+x^kg(k'),k+k'),
  \]
  showing that $g(k+k')=g(k)+x^kg(k')$.
\end{proof}

\begin{cor}
  For any $k\in\Z$ we have $g(km)=kg(m)$.
\end{cor}

\begin{proof}
  The action of $x^m$ on $\bigoplus_{\Z/m\Z}(\Z/p\Z)^n$ is trivial, and thus
  \[
    g(m+k)=g(m)+x^mg(k)=g(m)+g(k).
  \]
  The conclusion follows since $g(0)=0$.
\end{proof}

The preimage under $\varphi$ of $\bigoplus_{\Z/m\Z}(\Z/p\Z)^n$ %
is precisely $N\rtimes(m\Z)$. As $\varphi$ is onto, we must have
\[
  \psi(N)+\FF_p g(m)=\set{\psi(a)+g(km)\suchthat a\in N,k\in\Z}=\bigoplus_{\Z/m\Z}(\Z/p\Z)^n.
\]
In particular, the image $\psi(N)\sub\bigoplus_{\Z/m\Z}(\Z/p\Z)^n$ is a submodule of codimension at most $1$ (as $\FF_p$-vector spaces).

\begin{prop}
  The parameter $r$ from \eqref{eq:N-decomp} satisfies $r\ge n$.
\end{prop}

\begin{proof}
  Recall that $\bigoplus_{\Z/m\Z}(\Z/p\Z)^n\cong(R_p/(x^m-1)R_p)^n$ as $R_p$-modules; therefore we may view $\psi$ as an $R_p$-modules homomorphism $\psi\colon N\to(R_p/(x^m-1)R_p)^n$. By the above explanation, $\psi(N)$ is of codimension at most $1$ in $(R_p/(x^m-1)R_p)^n$. Since $\psi$ is a homomorphism of $R_p$-modules and $x^m-1$ annihilates its image, $\psi$ must split through a homomorphism $\overline{\psi}\colon N/(x^m-1)N\to (R_p/(x^m-1)R_p)^n$.

  Write $g_i=\gcd(f_i,x^m-1)$ for each $1\le i\le s$. By \eqref{eq:N-decomp}, we have
  \[
    N/(x^m-1)N \cong (R_p/(x^m-1)R_p)^r \times R_p/g_1R_p \times \cdots \times R_p/g_sR_p.
  \]
  Write also $N'=\overline{\psi}(N/(x^m-1)N)=\psi(N)$. On the one hand,
  \[
    \dim_{\FF_p} N' \le \dim_{\FF_p} N = rm+\sum_{i=1}^s\deg g_i\le rm+\sum_{i=1}^s \deg f_i.
  \]
  On the other hand, since $N'$ has codimension at most $1$ in $(R_p/(x^m-1)R_p)^n$, it follows that
  \[
    \dim_{\FF_p} N' \ge \dim_{\FF_p}(R_p/(x^m-1)R_p)^n - 1= nm - 1.
  \]
  Combining the inequalities yields
  \[
    (n-r)m \le \sum_{i=1}^s\deg f_i + 1.
  \]
  Since $m$ was chosen so that $m>\sum_{i=1}^s\deg f_i+1$, we must have $r\ge n$.
\end{proof}

\begin{cor}\label{cor:N-epi-Rn}
  There is an epimorphism $N\to R_p^n$ of $R_p$-modules.
\end{cor}

\subsection{Conclusion of the proof of \Tref{thm:main}}

We are now ready to finish the proof of the main theorem. Recall from \Rref{rem:sem-Rn-Z} that $R_p^n\rtimes\Z\cong\LL_{n,p}$, where the $\Z$-action on $R_p^n$ is induced by the action of $x\in R_p$ on $R_p^n$.

By \Cref{cor:N-epi-Rn}, there is an epimorphism of $R_p$-modules $\phi\colon N\to R_p^n$. Consider the homomorphism $f\colon\Gamma_0\to R_p^n\rtimes\Z\cong\LL_{n,p}$ given by $f(a,k)=(\phi(a),k)$. This is indeed a homomorphism, since
\begin{align*}
  f((a,k)(a',k'))&=f(a+x^ka',k+k')=(\phi(a+x^ka'),k+k')=\\
  &=(\phi(a)+x^k\phi(a'),k+k')=\\
  &=(\phi(a),k)(\phi(a'),k')=f(a,k)f(a',k').
\end{align*}
It is also clear that $f$ is onto, and thus there is an epimorphism $\Gamma_0\to R_p^n\rtimes\Z\cong\LL_{n,p}$. Composing $f$ with the natural projection $\Gamma\to\Gamma_0$ yields an epimorphism $\Gamma\to\LL_{n,p}$. In addition, $\Qu(\Gamma)=\Qu(\LL_{n,p})$ by assumption. We may therefore apply \cite[Theorem~3]{DFPR} and conclude that $\Gamma\cong\LL_{n,p}$, as required.

\bibliographystyle{plain}
\bibliography{prof_refs}

\end{document}